\documentclass[english]{amsart}

\usepackage{amsmath,amssymb,enumerate}
\usepackage{amscd,amsxtra,calc}

\usepackage[T1]{fontenc}
\usepackage[all]{xy}

\usepackage{babel}
\usepackage{amstext}
\usepackage{amsmath}
\usepackage{amsfonts}
\usepackage{latexsym}
\usepackage{ifthen}
\usepackage{verbatim}

\usepackage{xypic}
\xyoption{all}
\newcommand{\cal}{\mathcal}

\newtheorem{lemma1}{}[section]

\newenvironment{lemma}{\begin{lemma1}{\bf Lemma.}}{\end{lemma1}}

\newenvironment{theorem}{\begin{lemma1}{\bf Theorem.}}{\end{lemma1}}
\newenvironment{proposition}{\begin{lemma1}{\bf Proposition.}}{\end{lemma1}}
\newenvironment{corollary}{\begin{lemma1}{\bf Corollary.}}{\end{lemma1}}
\newenvironment{remark}{\begin{lemma1}{\bf Remark.}\rm}{\end{lemma1}}
\newenvironment{definition}{\begin{lemma1}{\bf Definition.}}{\end{lemma1}}

\newenvironment{conjecture}{\begin {lemma1}{\bf Conjecture.}}{\end{lemma1}}

\newenvironment{remarks}{\begin{lemma1}{\bf Remarks.}}{\end{lemma1}}

\newenvironment{remark*}{{\bf Remark.}}{}
\newenvironment{example*}{{\bf Example.}}{}

\newcommand{\Z}{\ensuremath{\mathbb{Z}}}
\newcommand{\C}{\ensuremath{\mathbb{C}}}
\newcommand{\N}{\ensuremath{\mathbb{N}}}
\newcommand{\PP}{\ensuremath{\mathbb{P}}}

\newcommand{\merom}[3]{\ensuremath{#1:#2 \dashrightarrow #3}}

\newcommand{\holom}[3]{\ensuremath{#1:#2  \rightarrow #3}}
\newcommand{\fibre}[2]{\ensuremath{#1^{-1} (#2)}}

\makeatletter
\ifnum\@ptsize=0  \fi \ifnum\@ptsize=2  \fi 

\newcommand\sO{{\mathcal O}}

\DeclareMathOperator*{\red}{red}

\newcommand{\chow}[1]{\ensuremath{\mathcal{C}(#1)}}

\newcommand{\upX}{\ensuremath{\tilde{X}}}
\newcommand{\upZ}{\ensuremath{\tilde{Z}}}
\newcommand{\upY}{\ensuremath{\tilde{Y}}}

\newcommand{\upA}{\ensuremath{\tilde{A}}}
\newcommand{\upF}{\ensuremath{\tilde{F}}}

\newcommand{\barX}{\ensuremath{\overline{X}}}
\newcommand{\barZ}{\ensuremath{\overline{Z}}}
\newcommand{\barY}{\ensuremath{\overline{Y}}}

\newcommand{\supp}[0]{\operatorname{Supp}}

\newcommand{\sing}[0]{\operatorname{Sing}}

\newcommand{\stab}[0]{\operatorname{stab}}

\newcommand{\mor}[0]{\operatorname{FinMor}} 

\newcommand{\Alb}[0]{\operatorname{Alb}} 
\newcommand{\uni}[0]{\operatorname{univ}}

\newcommand{\hilb}[0]{\mathcal H}
\newcommand{\univ}[0]{\mathcal U}

\title{Compact K\"ahler manifolds with compactifiable universal cover}
\date{\today}

\subjclass[2000]{32Q30, 14E30, 14J30}
\keywords{universal cover, Iitaka's conjecture}

\author{Beno\^it Claudon}
\author{Andreas H\"oring}

\begin{document}

\begin{abstract} 
Let $X$ be a compact K\"ahler manifold such that the universal cover admits a compactification.
We conjecture that the fundamental group is almost abelian and reduce this problem to a classical
conjecture of Iitaka.
\end{abstract}

\maketitle


\vspace{-1ex}

\section{introduction}


The aim of this paper is to study the following problem.

\begin{conjecture} \label{conjecturecover}
Let $X$ be a compact K\"ahler manifold with infinite fundamental group $\pi_1(X)$.
Suppose that the universal cover $\upX_{\uni}$ is a Zariski open subset $\upX_{\uni} \subset \barX$
of some compact complex manifold $\barX$. Then (after finite \'etale cover) there exists a locally trivial fibration
$X \rightarrow A$ with simply connected fibre $F$ onto a complex torus $A$. In particular we have $\upX_{\uni} \simeq F \times \C^{\dim A}$.
\end{conjecture}

This conjecture generalises Iitaka's classical conjecture claiming that a compact K\"ahler manifold $X$ uniformised by
$\C^{\dim X}$ is an \'etale quotient of a complex torus.
In a recent paper with J. Koll\'ar we studied this conjecture in the algebraic setting, i.e. under the additional 
hypothesis that $X$ and $\upX_{\uni}$ are quasi-projective. 
It turned out that the key issue is to show that the fundamental group is almost abelian and we established 
the following statement.

\begin{proposition}\label{propositionalmostabelian} \cite[Prop.1.3]{CHK11}
Let $X$ have the smallest dimension among all  normal, projective varieties
that have  an infinite,  quasi-projective, \'etale Galois cover $\tilde X\to X$
whose Galois group is not  almost abelian.

Then $X$ is smooth and its canonical bundle $K_X$ is nef but not semiample.
(That is, $(K_X\cdot C)\geq 0$ for every algebraic curve $C\subset X$
but $\sO_X(mK_X)$ is not generated by global sections for any $m>0$.)
\end{proposition}

By the {\it abundance conjecture} \cite[Sec.2]{Rei87b} the canonical bundle should always
be semiample if it is nef. We then proved
that in the algebraic case Conjecture \ref{conjecturecover} is implied by the abundance conjecture \cite[Thm.1.1]{CHK11}.

Since an infinite cover $\upX \rightarrow X$ is never an algebraic morphism, 
it is natural to look for an analogue of Proposition \ref{propositionalmostabelian} in the analytic category.
Although the existence of a compactification $\upX \subset \barX$
should already be quite restrictive we will see that the appropriate analytic analogue
of the quasiprojectiveness is the existence of a K\"ahler compactification. 

\begin{theorem} \label{theoremmain}
Let $X$ have the smallest dimension among all normal, compact K\"ahler spaces
that have  an infinite, \'etale Galois cover $\tilde X\to X$
whose Galois group $\Gamma$ is not almost abelian and such that there exists a K\"ahler compactification $\upX \subset \barX$.
Then $X$ is smooth, does not admit any Mori contraction, and $\upX$ is not covered by positive-dimensional compact subspaces.

In particular the fundamental group $\pi_1(X)$ is generically large, i.e. $\upX_{\uni}$ is not 
covered by positive-dimensional compact subspaces.
\end{theorem}

Even in the algebraic case, this statement gives some new information:
if $X$ is projective, the absence of Mori contractions implies that $K_X$ is nef.
Thus the ``minimal dimensional counterexample'' 
in Proposition \ref{propositionalmostabelian} has generically large fundamental group.
Note also that for a manifold with generically large $\pi_1$ the Conjecture \ref{conjecturecover}
simply claims that $X$ is an \'etale quotient of a torus. Thus we are reduced 
to Iitaka's conjecture which has been studied by several authors \cite{Nak99, Nak99b, CP08, a16}\footnote{
Apart from \cite{CP08} these papers do not really use that $\upX_{\uni} \simeq \C^{\dim X}$.}.

An important difference between the proof of Theorem \ref{theoremmain} and the arguments in \cite{CHK11} is that 
the natural maps attached to compact K\"ahler manifolds (algebraic reduction, reduction maps for covering families of algebraic cycles)
are in general not morphisms, as opposed to the classification theory of projective manifolds where we have
Mori contractions and, assuming abundance, the Iitaka fibration at our disposal.
Our key observation will be that if $G$ is a general fibre of the $\Gamma$-reduction $\gamma$ (cf. Definition \ref{definitiongammareduction}),
the aforementioned meromorphic maps are holomorphic.
Using this we obtain a strong dichotomy: up to replacing $\gamma$ by some factorisation
the general fibre $G$ is either projective or does not contain any positive-dimensional compact subspaces (cf. Theorem \ref{theoremminimalfactorisation}). 
In a similar spirit F. Campana shows in the Appendix (Theorem \ref{reduction conj S}) from a more general viewpoint that
Iitaka's conjecture has only to be treated for projective manifolds and simple compact K\"ahler manifolds, i.e. 
those which are not covered by positive-dimensional compact submanifolds.

If we try to avoid the K\"ahler assumption on $\barX$ we still obtain some information of birational nature:

\begin{proposition} \label{propositionnonkaehler}
Let $X$ have the smallest dimension among all normal, compact K\"ahler spaces
that have an infinite, \'etale Galois cover $\tilde X\to X$
whose Galois group $\Gamma$ is not almost abelian and such that there exists a compactification $\upX \subset \barX$.
Then $X$ is smooth and special in the sense of Campana \cite{Cam04}.
\end{proposition}

This proposition follows rather quickly from an orbifold version of the Kobayashi-Ochiai theorem (Theorem \ref{KO}).
By results of F. Campana and the first named author \cite[Thm.3.33]{Cam04}, \cite[Thm.1.1]{CC11}
the fundamental group of a special manifold of dimension at most three is almost abelian, 
so our counterexample (if it exists) would have $\dim X \geq 4$.

Let us finally note that once we have understood the fundamental group, the geometric statement 
in Conjecture \ref{conjecturecover} is not far away. 

\begin{theorem}\label{local triviality alb}
Let $X$ be a compact K\"ahler manifold whose universal cover $\upX_{\uni}$ 
admits a K\"ahler compactification $\upX_{\uni} \subset \barX$.
If the fundamental group of $X$ is almost abelian, 
the Albanese map of X is (up to finite \'etale cover) a locally analytically trivial fibration whose typical fibre $F$ is simply connected.
\end{theorem}

Since the proof of the corresponding statement in the algebraic setting \cite[Thm.1.4]{CHK11} relies on strong results of Hodge theory for 
birational morphisms which are unknown in the K\"ahler setting, our argument follows the lines of \cite{KP12}. 
Indeed if the fundamental group $\pi_1(X)$ is generically large, \cite[Thm. 16]{KP12}
implies that $X$ is isomorphic to its Albanese torus (even without any further assumption on $\barX$); 
see \cite{FK12} and Remarks \ref{final rem} for a discussion around this general case.

{\bf Acknowledgements.} This paper is a continuation of our work with J. Koll\'ar whom we thank for many
helpful communications and comments. The authors are supported by the ANR project CLASS.

\section{Notation and basic results}

Manifolds and complex spaces will always be supposed to be irreducible.

If $X$ is a normal complex space we denote by $\chow{X}$ its cycle space \cite{Ba75}.
We will use very often that if $X$ is a compact K\"ahler space, then the irreducible components of $\chow{X}$ are compact (Bishop's theorem, see \cite{Li78}).

Recall that a fibration  \holom{\varphi}{X}{Y} from a manifold $X$ onto a normal complex space $Y$ 
is almost smooth if the reduction $F_{\red}$ of every fibre is smooth and has the expected dimension.
In this case the complex space $Y$ has at most quotient singularities,
the local structure around $y \in Y$ being given by a finite representation of the fundamental group of $\pi_1(F_{\red})$ \cite[Prop.3.7]{Mol88}.
Thus there exists {\em locally} a finite base change $Y' \rightarrow Y$ such that the normalisation $X'$ of $X \times_Y Y'$
is smooth over $Y'$.

\begin{definition} \label{definitionalmostlocallytrivial}
We say that an almost smooth fibration \holom{\varphi}{X}{Y} 
is almost locally trivial with fibre $F$ if for every $y \in Y$ the fibration 
$X' \rightarrow Y'$ constructed above is locally trivial with fibre $F$.
\end{definition}

Note that while an almost locally trivial fibration is locally trivial in the neighbourhood of a generic point $y \in Y$, it is not true
that the reduction $F_{0, \red}$ of every fibre $F_0$ is isomorphic to $F$. For example if $F$ is a K3 surface with a fixed point free involution
$i: F \rightarrow F$ and $\Delta \subset \C$ the unit disc, then
$$
X := (F \times \Delta)/\hspace{-1ex}<i \times (z \mapsto -z)>
$$
has an almost locally trivial fibration $X \rightarrow \Delta$ with fibre $F$ and $F_{0, \red} \simeq F/\hspace{-1ex}<i>$. 

\begin{definition}\label{definitiongammareduction}\cite{Cam94,Kol93}
Let $X$ be a compact K\"ahler manifold and $\Gamma$ a quotient of the fundamental group $\pi_1(X)$.
There exists a unique almost holomorphic fibration\footnote{By unique we mean unique up to bimeromorphic equivalence of fibrations.} 
$$\merom{\gamma}{X}{\Gamma(X)}$$
with the following property: if $Z$ is a subspace with normalisation $Z' \rightarrow Z$ passing through a very general point $x\in X$ such that the natural
map $\pi_1(Z') \rightarrow \pi_1(X) \rightarrow \Gamma$ has finite image, then $Z$ is contained in the fibre through $x$.\\
This fibration is called the $\Gamma$-reduction of $X$ (Shafarevich map in the terminology of \cite{Kol93}).
\end{definition}

By definition the fundamental group $\pi_1(X)$ is generically large if the $\pi_1(X)$-reduction is a birational isomorphism \cite[Defn.1.7]{Kol93} (it corresponds to the case $\gamma d(X)=\dim(X)$ as defined in the Appendix).

\subsection{Compactifiable subsets}

\begin{definition} \label{definitionkaehlercompactifiable}
Let $\tilde X$ be a normal complex space.
We say that $\upX$ admits a K\"ahler compactification if there exists an embedding
$\upX \hookrightarrow \barX$ such that $\barX$ is a normal, compact K\"ahler space and $\upX$ is Zariski open in $\barX$.
\end{definition}

Let $\holom{\pi}{\upX}{X}$ be an infinite \'etale Galois cover with group $\Gamma$
such that $\upX$ admits a compactification $\upX \subset \barX$.
In \cite{CHK11} we assumed that the compactification $\barX$ is a projective variety 
and used the absence of algebraic $\Gamma$-invariant subsets to deduce important restrictions on the geometry of $X$.
While algebraic subsets $Z \subset \upX$ are always Zariski open in some projective subset $\barZ \subset \barX$,  
this is no longer true if we consider analytic subspaces $Z \subset \upX$. We have to restrict our considerations to a smaller class: 

\begin{definition} \label{definitioncompactifiable}
Let $\tilde X$ be a normal complex space 
such that we have a K\"ahler compactification $\upX \hookrightarrow \barX$.
A compactifiable subspace is a subset $Z \subset \tilde X$ such that there exists an analytic subspace $\bar Z \subset \bar X$,
an inclusion $Z \subset \bar Z$ and a subset $Z^* \subset Z$ such that
such that $Z^* \subset \barZ$ is Zariski open. 
A compactifiable subset is a finite union of compactifiable subspaces.
\end{definition}

\begin{remark*} 
In general the compactification $\barZ$ depends on the choice of $\barX$.
\end{remark*}

The following lemma explains why compactifiable subsets are interesting in our context.

\begin{lemma} \label{lemmacompactifiable}
Let $\tilde X$ be a normal complex space 
such that we have a K\"ahler compactification $\upX \hookrightarrow \barX$.
Let $\tilde{\hilb}$ be an irreducible component of $\chow{\upX}$
and $\tilde{\univ}$ be the universal family over $\tilde{\hilb}$.
Let $\holom{\tilde q}{\tilde{\univ}}{\tilde{\hilb}}$ and $\holom{\tilde p}{\tilde{\univ}}{\tilde  X}$ be the natural morphisms.
Then $\tilde  p(\tilde{\univ})$ is a compactifiable subset.
\end{lemma}

The idea of the proof is quite simple: $\tilde{\hilb}$ admits a natural compactification in $\chow{\barX}$, the corresponding
universal family compactifies $\tilde  p(\tilde{\univ})$. Since we use the statement several times we give the details
of the proof:

\begin{proof}
We set $D:=\barX \setminus \upX$. We have a natural inclusion $\chow{\upX} \hookrightarrow \chow{\barX}$
and we choose an irreducible component $\overline{\hilb}$ that contains the image of $\tilde{\hilb}$.
Denote by  $\overline{\univ}$ the universal family over $\overline{\hilb}$, endowed with the reduced structure,
and by $\holom{\bar q}{\overline{\univ}}{\overline{\hilb}}$ resp. $\holom{\bar p}{\overline{\univ}}{\barX}$ the natural morphisms.
We summarize the construction in a commutative diagram:
\begin{equation} \label{diagramcompactification}
\xymatrix{
\overline{\mathcal U} \ar[rr]^{\bar p}  \ar[dd]^{\bar q}  & & \barX = \upX \sqcup D
\\
& \mathcal{\tilde U} \ar @{_{(}->}[lu]  \ar[r]^{\tilde p}  \ar[d]^{\tilde q} &  \upX \ar @{_{(}->}[u] 
\\
\overline{\mathcal H} & \tilde{\mathcal H} \ar @{_{(}->}[l]  &
}
\end{equation}

The complex space $\barX$ being compact K\"ahler, $\overline{\hilb}$ and $\overline{\univ}$ are compact, hence by Remmert's proper mapping theorem $\bar p(\overline{\univ})$ 
is a finite union of analytic subspaces of $X$. Moreover $\bar q(\fibre{\bar p}{D})$ is a finite union of analytic spaces
and the inclusion $\bar q(\fibre{\bar p}{D}) \subset \overline{\hilb}$ is strict since $\bar q(\fibre{\bar p}{D})$
is disjoint from $\tilde{\hilb}$. Since $\tilde{\hilb}$ is an irreducible component of $\chow{\upX}$ this actually shows
that $\overline{\hilb} = \tilde{\hilb} \sqcup \bar q(\fibre{\bar p}{D})$.
In particular $\overline{\hilb}$ is unique and the Zariski closure of $\tilde{\hilb} \subset \chow{\barX}$.
Note now that $\bar p(\fibre{\bar q}{\bar q(\fibre{\bar p}{D})})$ is a finite union of analytic subspaces
of  $\bar p(\overline{\univ})$  which for reasons of dimension does not contain any irreducible component
of $\bar p(\overline{\univ})$. Thus
$$
Z^*:= \bar p(\overline{\univ})  \setminus \bar p(\fibre{\bar q}{\bar q(\fibre{\bar p}{D})})
$$ 
is dense and Zariski open in $\bar p(\overline{\univ})$, moreover we have an inclusion
$$
Z^* \subset \bar p(\tilde{\univ}) \subset \bar p(\overline{\univ}).
$$
Since $\bar p(\tilde{\univ})= \tilde  p(\tilde{\univ})$ this proves the statement.
\end{proof}

For certain irreducible components of $\chow{\upX}$ we can say more:

\begin{corollary}\label{corollaryexceptionallocus}
In the situation of the lemma above suppose moreover that $\tilde{\univ}$ is irreducible and $\tilde p$ is onto and generically finite.
Set
$$
\tilde B =  \{ \tilde x \in  \upX \ | \ \dim \fibre{\tilde p}{\tilde x} > 0 \}
$$ 
Then $\tilde B$ is a compactifiable subset.
\end{corollary}

\begin{proof}
As in the proof of the preceding lemma we consider the compactification $\tilde{\hilb} \subset \overline{\hilb}$
and the corresponding compactification of universal families $\tilde{\univ} \subset \overline{\univ}$.

{\em 1st case. $\tilde p$ is bimeromorphic.}
The morphism $\bar p$ is onto and bimeromorphic and we denote by $\bar B$ the image of its exceptional locus,
which is of course a finite union of analytic subspaces.
We have an inclusion
$\tilde B \subset \bar B$ and we are done if we show that $\tilde B=\bar B \cap \tilde X$.
To see this take  $\tilde x \in \upX$ such that $\tilde x \not \in \tilde B$, then 
$\bar q(\fibre{\bar p}{\tilde x})$ is the union of $\tilde q(\fibre{\tilde p}{\tilde x})$ and the cycles parametrised
by $\bar q(\fibre{\bar p}{D})$ passing through $\tilde x$. Yet $\tilde q(\fibre{\tilde p}{\tilde x})$ is a singleton
in $\tilde{\mathcal H}$, hence
disjoint from $\bar q(\fibre{\bar p}{D})$. Moreover $\bar q(\fibre{\bar p}{\tilde x})$ is connected by Zariski's main theorem, so
$\fibre{\bar p}{\tilde x}$ is a unique point. This proves the claim.

{\em 2nd case. $\tilde p$ generically finite.}
The morphism $\bar p$ is onto and generically finite, and we denote by $\bar p_{St}: \overline{\univ} \rightarrow \barX_{St}$
and $\holom{\mu}{\barX_{St}}{\barX}$ the Stein factorisation. Note that $\barX_{St}$ contains a Zariski open dense subset
$\upX_{St}:=\fibre{\mu}{\upX}$ and 
$$
\fibre{\mu}{B} = \{ \tilde x \in  \upX_{St} \ | \ \dim \fibre{\tilde p_{St}}{\tilde x} > 0 \}.
$$ 
The first case shows that the right hand side is compactifiable, hence $B$ is compactifiable.
\end{proof}

\subsection{The $\Gamma$-reduction}

Let $Y$ be a normal K\"ahler space 
such that we have a K\"ahler compactification $Y \hookrightarrow \barY$,
and let $U\to V$  be a flat, proper K\"ahler fibration.
In \cite[Section 2]{CHK11} we introduced the moduli spaces 
$\mor(U/V, Y, d)$ of finite  morphisms
$\phi:U_v\to X$ where $U_v$ is a fibre of $U \rightarrow V$ and $d$ is the degree of the graph of $\phi$ with respect to some fixed
K\"ahler forms on $U$ and $Y$.
These spaces of morphisms are Zariski open sets in the relative cycle spaces $\chow{U\times \bar Y/V}$, so we know by
Bishop's theorem that for bounded degree $d$ there are only finitely many irreducible components. Moreover
we have seen in Lemma \ref{lemmacompactifiable} that the images of universal families are compactifiable subsets of $Y$.
We can now argue as in \cite[Lemma 2.4]{CHK11} to prove the following:

\begin{lemma}\label{mainlemma}
Let $X$ be a normal, compact K\"ahler space and $\pi:\tilde X\to X$ 
 an infinite \'etale Galois cover
with group $\Gamma$ such that $\upX$ admits a K\"ahler compactification $\upX \subset \barX$. 
Let $X^0\subset X$ be a dense, Zariski open subset and $g^0:X^0\to Z^0$  a 
flat, proper fibration with general fiber $F$ such that $\pi$
induces a finite covering $\tilde F\to F$.  Let  
$\tilde g^0: \tilde X^0\to \tilde Z^0$
be the corresponding flat, proper fibration with general fiber $\tilde F$.
Then (at least) one of the following holds:
\begin{enumerate}
\item $\tilde g^0$  extends to a locally trivial, $\Gamma$-equivariant fibration
$\tilde g:\tilde X\to \tilde Z$ with fibre $\tilde F$, or
\item $\tilde X$ contains a compactifiable $\Gamma$-invariant subspace
that is disjoint from a general fiber of  $\tilde g^0$.
\end{enumerate}
\end{lemma}

Since  $\tilde g:\tilde X\to \tilde Z$
is  $\Gamma$-equivariant, the $\Gamma$-action on $\tilde X$
descends to a $\Gamma$-action on $\tilde Z$.
If $\tilde F$ has no fixed point free automorphisms, the
$\Gamma$-action on $\tilde Z$ is fixed point free,
but in general it can have finite stabilizers. Thus $g^0$ only extends to
a fibration $g:X\to Z$ that is almost locally trivial.

\begin{corollary} \label{corollaryfibration}
Let $X$ be a compact K\"ahler manifold and $\pi:\tilde X\to X$ 
 an infinite \'etale Galois cover
with group $\Gamma$ such that $\upX$ admits a K\"ahler compactification $\upX \subset \barX$. 
Suppose that $\upX$ does not contain any $\Gamma$-invariant compactifiable subsets.

Then the $\Gamma$-reduction is an almost locally trivial (cf. Definition \ref{definitionalmostlocallytrivial}) holomorphic map \holom{\gamma}{X}{\Gamma(X)}
and the corresponding fibration  \holom{\tilde \gamma}{\upX}{\tilde{\Gamma(X)}} is $\Gamma$-equivariant and locally trivial.

If there exists an almost holomorphic map $\varphi_F: F \dashrightarrow W$ with general fibre $G$, 
this map extends to an almost locally trivial holomorphic map \holom{\varphi}{X}{Y}
and the corresponding fibration  \holom{\tilde \varphi}{\upX}{\upY} is $\Gamma$-equivariant and locally trivial.
We call $\varphi$ a factorisation of the $\Gamma$-reduction with fibre $G$.
\end{corollary}

\begin{remark} \label{remarkarrangethings}
Let $G$ be a general $\varphi$-fibre. By definition of the $\Gamma$-reduction the natural
map $\pi_1(G) \rightarrow \pi_1(X) \rightarrow \Gamma$ has finite image $G_\Gamma$, so $\pi$
induces a finite \'etale cover $\tilde G \rightarrow G$ where $\tilde G$ is a general $\tilde \varphi$-fibre.
Up to replacing $X$ by the finite \'etale cover $X' \rightarrow X$ with Galois group $G_\Gamma$
and $\Gamma$ by 
$$
\Gamma' := \mbox{im} (\pi_1(X') \rightarrow \pi_1(X) \rightarrow \Gamma)
$$
we can suppose that $G_\Gamma$ is trivial, hence $\tilde G \simeq G$.
Since $\Gamma' \subset \Gamma$ has finite index, $\Gamma$ is almost abelian if and only if this holds for $\Gamma'$.
\end{remark}

\begin{corollary} \label{corollaryrigid}
In the situation of Corollary \ref{corollaryfibration}, let \holom{\varphi}{X}{Y}
be a factorisation of the $\Gamma$-reduction with fibre $G$.
Then $G$ does not contain any rigid subspaces, i.e. there is no subspace
$Z \subset G$ such that for all $m \in \N$ the Chow space has pure dimension $0$
in the point $[mZ]$.
\end{corollary}

\begin{proof}
We argue by contradiction and suppose that such a subspace $Z$ exists. 
For $y \in Y$ general, the fibration $\varphi$ is locally trivial near $y$, i.e. 
there exists an analytic neighbourhood $y \in U \subset Y$ such that
$\fibre{\varphi}{U} \simeq U \times G$. In particular the relative Chow space $\chow{\fibre{\varphi}{U}/U}$ is isomorphic to a product $U \times \chow{G}$,
so there exists a unique irreducible component of $\chow{X}$ parametrising deformations of $Z$ in $X$ 
that dominates $Y$. The reduction of this irreducible component is isomorphic to $Y$.
 
Note now that $\tilde Z:=\fibre{\pi}{Z}$ is a finite union of subspaces in $\tilde G:=\fibre{\pi}{G}$ which are rigid: otherwise
their deformations would induce a deformation of some multiple of the cycle $[Z]$.
Thus the deformations of $\tilde Z$ in $\tilde X$ correspond to an irreducible component of $\chow{\upX}$
whose reduction is isomorphic to $\tilde Y$. By Lemma \ref{lemmacompactifiable} the deformations of $\tilde Z$ 
cover a compactifiable subset of $\tilde X$. Moreover it is $\Gamma$-invariant since it is the $\pi$-preimage of the locus
covered by deformations of $Z$. Thus we have constructed a $\Gamma$-invariant compactifiable subset, a contradiction.
\end{proof}

\begin{remark} \label{remarkbimeromorphic}
If $\mu: G \rightarrow G'$ is a bimeromorphic morphism onto a K\"ahler space, an irreducible component $Z \subset G$ of the $\mu$-exceptional locus
is rigid. Indeed if $Z$ has dimension $d$, then $Z \cdot \mu^* \omega^d= \mu(Z) \cdot \omega^d=0$ where $\omega$ is a K\"ahler form on $G'$.
Thus if $Z'$ is a small deformation of the cycle $[mZ]$, then $0 = Z' \cdot \mu^* \omega^d= \mu(Z') \cdot \omega^d$, so $\mu(Z')$ has dimension strictly smaller than $d$. Thus $Z'$ is contained in the $\mu$-exceptional locus. The complex space $Z$ being an irreducible component of this locus,
we have $\supp(Z)=\supp(Z')$.  
\end{remark}

\begin{corollary} \label{corollaryalgebraicreduction}
In the situation of Corollary \ref{corollaryfibration}, let \holom{\varphi}{X}{Y}
be a factorisation of the $\Gamma$-reduction with fibre $G$.
Then the algebraic reduction $G \dashrightarrow A(G)$ is holomorphic.
\end{corollary}

\begin{proof}
By a theorem of Campana \cite[Cor.10.1]{Cam88}, 
the general fibres of the algebraic reduction define an irreducible component of $\chow{G}$, i.e. 
there exists an irreducible component $\hilb_G$ of $\chow{G}$
such that the general point corresponds to a general fibre of the algebraic reduction. 
In particular if $\univ_G$ is the universal family over $\hilb_G$, the natural
morphism $\univ_G \rightarrow G$ is onto and bimeromorphic.
Using the local triviality of $\varphi$ as in the proof of Corollary \ref{corollaryrigid} above, we obtain
an irreducible component $\hilb$ of $\chow{X}$ such that the natural map $\holom{p}{\univ}{X}$
is onto and bimeromorphic. If the image $B$ of the $p$-exceptional locus is empty we are obviously done.  

Suppose now that this is not the case. Then we have
an irreducible component $\tilde{\hilb}$ of $\chow{\upX}$ such that the natural map $\holom{\tilde p}{\tilde \univ}{\upX}$
is onto and bimeromorphic, moreover the set 
$$
\tilde B := \fibre{\pi}{B}  =  \{ \tilde x \in  \upX \ | \ \dim \fibre{\tilde p}{\tilde x} > 0 \}
$$ 
is $\Gamma$-invariant and compactifiable by Corollary \ref{corollaryexceptionallocus}. Again a contradiction to our assumption.
\end{proof}

\begin{remark*}
Note that our proof heavily relies on the property that the general fibres of the algebraic reduction 
define an {\em irreducible component} of the Chow space. This holds for any almost holomorphic
fibration, but fails in general: if \merom{g}{\PP^n}{\PP^{n-1}} is the projection from a point $x$, the fibres correspond to lines through $x$,
so they define a proper subset of the irreducible component of $\mathcal C(\PP^n)$ parametrising lines.
\end{remark*}

We can now prove the main statement of this section:

\begin{theorem} \label{theoremminimalfactorisation}
Let $X$ be a compact K\"ahler manifold and $\pi:\tilde X\to X$ 
 an infinite \'etale Galois cover
with group $\Gamma$ such that $\upX$ admits a K\"ahler compactification $\upX \subset \barX$. 
Suppose that $\upX$ does not contain any $\Gamma$-invariant compactifiable subsets.
Let $\varphi:X \rightarrow Y$ be a factorisation of the $\Gamma$-reduction such that the fibre $G$
has minimal, but positive dimension. Then (up to replacing $X$ by a finite \'etale cover) 
the manifold $G$ is either projective or does not contain any positive-dimensional subspaces.
\end{theorem}

\begin{proof}
Suppose that $G$ is not projective, i.e. $a(G)<\dim G$. 

{\em 1st step. Suppose that $a(G)>0$.}
Then by Corollary \ref{corollaryalgebraicreduction} the algebraic reduction $G \dashrightarrow A(G)$
is holomorphic. By Corollary \ref{corollaryfibration} this induces a factorisation of the $\Gamma$-reduction
whose fibres have strictly smaller dimension, a contradiction. 

{\em 2nd step. Suppose that $G$ is covered by positive-dimensional subspaces.}
Note first that 
a compact complex manifold $G$ with $a(G)=0$ contains only finitely many divisors \cite{FF79}.
Thus by Corollary \ref{corollaryrigid} the manifold $G$ contains no divisors since these would be rigid.
Let now $\hilb_G \subset \chow{G}$ be an irreducible component
of the cycle space parametrising a covering family of positive-dimensional subspaces of maximal dimension.
Then by Lemma \ref{lemmafinite} below
the map from the universal family $p_G: \univ_G \rightarrow G$ is generically finite.
Arguing as in the proof of Corollary \ref{corollaryalgebraicreduction} we construct irreducible components
$\hilb \subset \chow{X}$ and $\tilde{\hilb} \subset \chow{\upX}$ such that the restriction to $G$ is $\hilb_G$.

Thus the maps from the universal families $p: \univ \rightarrow X$
and $\tilde p: \tilde{\univ} \rightarrow \upX$ are onto and generically finite, 
hence the set 
$$
\tilde B  =  \{ \tilde x \in  \upX \ | \ \dim \fibre{\tilde p}{\tilde x} > 0 \}
$$ 
is $\Gamma$-invariant and compactifiable by Corollary \ref{corollaryexceptionallocus}. 
By our hypothesis $\tilde B$ is empty, so the map $p_G$ is finite. Since $G$ is smooth and does not contain
any divisors, we see by purity of the branch locus that (up to replacing $\univ_G$ by its normalisation) the map
$p_G$ is \'etale. Hence $p$ and $\tilde p$ are \'etale and  $\tilde{\univ}$ can be compactified
by the universal family over a compactification $\tilde{\hilb} \subset \overline{\hilb}$. Thus up to replacing
$X$ by the finite \'etale cover $\univ \rightarrow X$ we can suppose that $p_G$ is an isomorphism.
Yet then $G \simeq \univ_G$ admits a natural fibration $q_G: \univ_G \rightarrow \hilb_G$,
so we get again a fibration with fibres of strictly smaller dimension, a contradiction. 

{\em 3rd step. $G$ has no positive-dimensional subspaces.}
Let $Z \subset G$ be a positive-dimensional subspace of maximal dimension and take $m \in \N$ arbitrary.
Then the Chow scheme has dimension zero in the point $[mZ]$: indeed if $\hilb_G$ is an irreducible
component passing through $[mZ]$, the map from the universal family $\univ_G \rightarrow G$ is not onto
by  the 2nd step. By maximality of the dimension the image has dimension equal to $\dim Z$, so $\hilb_G$
has dimension zero. Thus $Z$ is rigid, which is excluded by Corollary \ref{corollaryrigid}.
\end{proof}

\begin{lemma} \label{lemmafinite}
Let $G$ be a compact K\"ahler manifold such that $a(G)=0$. Let $\hilb_G \subset \chow{G}$ be an irreducible component
of the cycle space parametrising a covering family of positive-dimensional subspaces which have maximal dimension $m$, i.e.
there is no covering family of subspaces with dimension strictly larger than $m$. 
Let $\univ_G$ be the universal family over $\hilb_G$, and denote by 
$p_G: \univ_G \rightarrow G$ and $q_G: \univ_G \rightarrow \hilb_G$ the natural morphisms.
Then $p_G$ is generically finite.
\end{lemma}

\begin{proof}
Since $G$ contains only finitely many divisors \cite{FF79},
we have $m<\dim G-1$.
We argue by contradiction, and suppose that the general $p_G$-fibre is positive-dimensional.
Then for $g \in G$ general, the analytic set $q_G(\fibre{p_G}{g})$ is positive-dimensional and Moishezon
by \cite[Cor.1]{Cam80}. In particular $q_G(\fibre{p_G}{g})$ is covered by compact curves.
Choose an irreducible curve $C_g \subset q_G(\fibre{p_G}{g})$, then
$p_G(\fibre{q_G}{C_g})$ has dimension $m+1<\dim G$. Since $g \in G$ is general we can construct in this way a covering
family of strictly higher dimension, a contradiction.
\end{proof}

\subsection{Fibre bundles} \label{subsectionfibrebundles}

Let us recall the following facts about the automorphism group of a compact K\"ahler manifold $G$ \cite{F78,Li78}. The identity component 
$\mathrm{Aut}^0(G)$ of the complex Lie group $\mathrm{Aut}(G)$ has a description in terms of the Albanese torus of $G$. 
Consider the natural map
$$\mathrm{Aut}^0(G)\longrightarrow \mathrm{Aut}^0(\Alb(G))\simeq\Alb(G)$$
induced by the Albanese mapping. The kernel of this morphism is a linear algebraic group and the image is a subtorus of $\Alb(G)$; in particular, if $G$ is not covered by rational curves, $\mathrm{Aut}^0(G)$ is a compact group, isogeneous to a subtorus of $\Alb(G)$. If $\mathrm{Aut}^0(G)$
is a point, fibre bundles with fibre $G$ can be easily described.

\begin{lemma} \label{lemmadiscrete} \cite[Cor.4.10]{F78}
Let $\holom{\varphi}{X}{Y}$ be a proper fibre bundle with typical fibre a manifold $G$.
Suppose that the automorphism group $\mbox{Aut}(G)$ is discrete and that the fibration $\varphi$
is K\"ahler, i.e. there exists a two-form $\omega$ on $X$ such that the restriction $\omega|_G$
is a K\"ahler form. Then there exists a finite \'etale base change $Y' \rightarrow Y$
such that $X \times_Y Y' \simeq Y' \times G$.
\end{lemma}

\begin{proof}[Sketch of proof.] 
The K\"ahler assumption on the morphism implies that the structure group of the fibre bundle can be reduced to $\mathrm{Aut}(F,[\omega|_F])$, the group of automorphisms of $F$ which preserves the cohomology class of $\omega|_F$. By Fujiki-Lieberman 
this group contains $\mathrm{Aut}^0(G)=\{1\}$ as a finite index subgroup and is thus finite. Thus the image of the monodromy presentation
is finite, hence trivial after finite \'etale base change.
\end{proof}

\section{Proofs of the main results}

\begin{proof}[Proof of Theorem \ref{theoremmain}] 
We claim that $\upX$ has no nontrivial, 
compactifiable subset $\tilde W$  invariant under a
finite index subgroup  $\Gamma'\subset \Gamma$. 
Note first that if such  a $\tilde W$ exists, we can suppose it to be analytic: otherwise replace it by
$\upX \cap \overline W$ where $\overline W \subset \barX$ is a compactification.
If we denote by $\tilde W_i$ the irreducible components of $\tilde W$, each of them
is invariant under a
 finite index subgroup  $\Gamma_i\subset \Gamma$.
Taking the  normalization $\tilde W^n_i$, 
we would get a smaller dimensional example
$W^n_i:=\tilde W^n_i/\Gamma_i$ as in Theorem \ref{theoremmain}; a contradiction.

Since $\upX_{\sing} \subset \upX$ is compactifiable and $\Gamma$-invariant,
we conclude that $X$ is smooth. We claim that $X$ is not uniruled, i.e. it is not covered by rational curves: 
otherwise we can consider the MRC-fibration $\merom{\varphi}{X}{Z}$. Since the general $\varphi$-fibre $G$ is rationally connected,
hence simply connected, the MRC-fibration is a factorisation of the $\Gamma$-reduction.
By Corollary \ref{corollaryfibration} it extends to an almost locally trivial holomorphic map  $\holom{\varphi}{X}{Z}$ 
and the corresponding fibration $\holom{\tilde \varphi}{\upX}{\upZ}$ is $\Gamma$-equivariant and locally trivial with fibre $G$.
Note that $G$ does not admit fixed point free actions by any finite group: 
the \'etale quotient
would also be rationally connected, so simply connected.
Therefore the stabilizer $\stab_{\Gamma}(F_z)$ is trivial for
every $\tilde \varphi$-fiber $F_z$.
Hence the $\Gamma$-action  descends to a free $\Gamma$-action on $\tilde Z$;
a contradiction to the minimality of the dimension of $X$.

Arguing by contradiction we will now prove that the $\Gamma$-reduction $\gamma$ is an isomorphism.
If this is not the case we know by Theorem \ref{theoremminimalfactorisation} that 
there exists a factorisation $\varphi: X \rightarrow Y$ of the $\Gamma$-reduction
such that the fibre $G$ is either projective or without positive-dimensional
compact subspaces. Let us consider the corresponding locally trivial fibration
$\tilde \varphi: \upX \rightarrow \upY$ with fibre $G$ (cf. Remark \ref{remarkarrangethings}). 

{\em 1st case. $\mbox{Aut}^0(G)$ is a point.}
The structure group of the fiber bundle  $\tilde \varphi: \upX \rightarrow \upY$ is discrete, so by Lemma \ref{lemmadiscrete}
we can suppose (after finite \'etale cover) that $\upX \simeq \upY \times G$. The $\Gamma$-action
on $\upX$ commutes with the projection on $\upY$ and the group $\mbox{Aut}(G)$ is discrete,
so $\Gamma$ acts diagonally on the product $\upY \times G$.
Consider now the $\tilde \varphi$-section $\upY \times g$ for some $g \in G$: its orbit $\cup_{f \in \Gamma} f(\upY \times g)$ is a finite union
of sections of the form $\upY \times g'$. The complex space $\upY$ has a natural compactification in $\chow{\barX}$, since
it parametrises the $\tilde \varphi$-fibres. Thus we have constructed a $\Gamma$-invariant compactifiable subset of $\upX$, a contradiction
to the minimality of $X$.

{\em 2nd case. $\mbox{Aut}^0(G)$ has positive dimension.}
Since $G$ is not uniruled we know by Section \ref{subsectionfibrebundles} that
the group $\mbox{Aut}^0(G)$ is 
isogeneous to a subtorus of the Albanese torus of $G$, in particular we have
$q(G)>0$ . We claim that in this case $G$ is a torus.
Assuming this for the time being, let us see how to conclude: 
since $\varphi$ is almost locally trivial we can apply \cite[\S 6]{CP00}, \cite[Prop.4.5]{Fuj83}: after a finite \'etale cover $X' \rightarrow X$
we can suppose that $q(X)=q(Y)+\dim G$. 
Since every $\varphi$-fibre is irreducible, the
Albanese map $\alpha_X: X \rightarrow \Alb(X)$ maps
each $\varphi$-fibre isomorphically onto a fibre of the locally trivial fibration $\varphi_*: \Alb(X) \rightarrow \Alb(Y)$.
By the universal property of the fibre product we have a commutative diagram
$$
\xymatrix{
\Alb(X) \times_{\Alb(Y)} Y   \ar[rd]_\psi 
& X  \ar[l] \ar[r]^{\alpha_X} \ar[d]_{\varphi} & \Alb(X) \ar[d]^{\varphi_*}
\\
& Y \ar[r]^{\alpha_Y}  & \Alb(Y)
}
$$
The map $\psi$ is the pull-back of $\varphi_*$ by the fibre product, so it is a locally trivial fibration. The base $Y$ is normal,
so the total space $\Alb(X) \times_{\Alb(Y)} Y$ is normal.
By what precedes the morphism $X \rightarrow \Alb(X) \times_{\Alb(Y)} Y$ is bimeromorphic and finite, hence
an isomorphism by Zariski's main theorem. In particular $\varphi=\psi$ is smooth and locally trivial.
Hence the $\Gamma$-action on $\upX$ descends to a free action on $\upY$, i.e.
$X$ is not a minimal counterexample.

{\em Proof of the claim.}
This is clear if $G$ has no compact positive-dimensional subspaces,
so we can suppose that $G$ is projective. Since $q(G) \neq 0$ 
the Albanese map is non-trivial.
By minimality of the factorisation we see that the Albanese map $G \rightarrow A(G)$ is generically finite
onto its image. Since $G$ admits no birational map by Remark \ref{remarkbimeromorphic},
it is actually finite onto its image.
Moreover $G$ is not of general type since  $\mbox{Aut}^0(G)$ has positive dimension.
If we have $\dim G>\kappa(G)>0$ we know by \cite{Kaw81} that $G$ admits a fibration by positive-dimensional abelian varieties, 
a contradiction to the minimality of the factorisation.
Thus we have $\kappa(G)=0$ and $G$ is an abelian variety. This proves the claim.

Let us finally show that $X$ does not admit any Mori contraction, i.e. does not admit any morphism with connected fibres
\holom{\mu}{X}{X'} onto a normal complex space $X'$ such that $-K_X$ is $\mu$-ample. Since $X$ is not uniruled, $\mu$ would necessarily be bimeromorphic. Moreover $\mu$ is a projective morphism since it is polarised by $-K_X$.
In particular the Ionescu-Wi\'sniewski inequality  \cite[Thm.0.4]{Ion86}, \cite[Thm.1.1]{Wis91} applies and shows
that if $E$ is an irreducible component of the exceptional locus and $F$ a general fibre of $E \rightarrow \mu(E)$, then
one has $\dim E+\dim F \geq \dim X$. Arguing as in the \cite[Lemma 2.5]{CHK11} we can now prove that $\fibre{\pi}{E}$
is a $\Gamma$-invariant compactifiable subset of $\upX$, contradicting
the minimality of $X$.
\end{proof}

The proof of Proposition \ref{propositionnonkaehler} relies on 
the following generalisation of the Kobayashi-Ochiai theorem to fibrations of general type.

\begin{theorem} \cite[Thm.2]{KO75}\cite[Thm.8.2]{Cam04}\label{KO}
Let $X$ be a compact K\"ahler manifold, $\barX$ be a complex manifold, and let $B \subset \barX$ be a proper closed analytic subset. Let \merom{\pi}{\barX \setminus B}{X} be a nondegenerate meromorphic map, \emph{i.e.} such that the tangent map
$T_{\barX \setminus B} \rightarrow T_X$  is surjective at least one point $v \in \barX \setminus B$. 
Let us finally consider \merom{g}{X}{Y} a general type fibration\footnote{We refer to \cite{Cam04} for the basic definitions
of the orbifold theory.} defined on $X$. 
Then $f=g\circ \pi$ extends to a meromorphic map $\barX \dashrightarrow Y$.
\end{theorem}

\begin{proof}[Proof of Proposition \ref{propositionnonkaehler}]
As in the proof of Theorem \ref{theoremmain} the minimality condition implies that $\upX$ does not contain any
$\Gamma$-invariant compactifiable subset.
Since the singular locus $\upX_{\sing}$ is $\Gamma$-invariant  and naturally compactified by $\barX_{\sing}$, 
we see that $X$ is smooth. 

Let us now argue by contradiction and suppose that $X$ is not special.
Then the core fibration \merom{c_X}{X}{C(X)} \cite[section 3]{Cam04} is not trivial, i.e. the base has dimension at least one. 
Moreover it is a general type fibration, so by Theorem \ref{KO} above the composed map $\merom{c_X \circ \pi}{\upX}{C(X)}$
extends to a meromorphic map $\merom{\bar c_X}{\barX}{C(X)}$.
Up to replacing $\barX$ by a suitable bimeromorphic model, we can assume that $\bar c_X$ is holomorphic.
Since $\barX$ is proper, a general $\bar c_X$-fibre $\bar F:=\fibre{\bar c_X}{y}$ has finitely many irreducible components, each of them of dimension $\dim X-\dim C(X)$ and not contained in $\barX \setminus \upX$. 
Thus if $F:=\fibre{c_X}{y}$ is the corresponding $c_X$-fibre, then $\fibre{\pi}{F}=\bar F \cap \upX$.
Thus $\fibre{\pi}{F}$ is a $\Gamma$-invariant compactifiable subset of $\upX$, a contradiction.
\end{proof}

As mentioned in the introduction, the proof given in \cite{CHK11} of the local triviality of the Albanese map does not apply verbatim in this non algebraic setting. We need the following purely topological result.
\begin{theorem}\label{surjectivity} \cite[Thm. 14]{KP12}
Let $\holom{f}{Z}{A}$ be a map between compact analytic spaces, the universal cover $\upA_{univ}$ being contractible. Let us denote by $\tilde{Z}$ the induced cover of $Z$. If $\tilde{Z}$ has the homotopy type of a compact metric space, then $f$ is surjective.
\end{theorem}
With this in mind, we can prove Theorem \ref{local triviality alb} in the spirit of \cite{KP12}.

\begin{proof}[Proof of Theorem \ref{local triviality alb}]
To begin with, let us recall that classical arguments (see \cite{CHK11}) show that the Albanese map is always a fibration ($i.e.$ a surjective map with connected fibres) when $\upX_{univ}$ is a Zariski open subset of a compact complex manifold. Since $\pi_1(X)$ is supposed to be almost abelian, we can also assume that
$$\holom{\alpha_X}{X}{A:=\Alb(X)}$$
is a fibration which induces an isomorphism at the level of fundamental groups. Let us consider now $F$ a smooth fibre of $\alpha_X$ and denote by $f$ its homology class. Let us introduce $\mathcal{Z}_F(X)$ the unique irreducible component of $\chow{F\times X}$ which contains the graph of the embeddings $j:F\hookrightarrow X$ whose homology class is fixed: $j_*[F]=f$ in $\mathrm{H}_*(X,\Z)$. The homology class being fixed, there is a natural map:
$$\holom{\alpha_*}{\mathcal{Z}_F(X)}{A}$$
(the complex space $\mathcal{Z}_F(X)$ should be thought as the set of fibres of $\alpha_X$ which are isomorphic to $F$). Our aim is to apply Theorem \ref{surjectivity} to show that $\alpha_*$ is surjective; we have to prove some topological finiteness of the fibre product:
$$\widetilde{\mathcal{Z}_F(X)}\longrightarrow\upA_{univ}.$$

Since $\alpha_X$ induces an isomorphism on the $\pi_1$, the induced map $\upX_{\uni} \rightarrow \upA_{\uni}$ 
between the universal covers is proper. We can then choose $\upF$ any lifting of $F$ (it is a compact submanifold of $\upX_{univ}\subset\barX$) and perform the same construction on $\barX$: we denote by $\mathcal{Z}_{\upF}(\barX)$ the complex space of embeddings of $\upF$ into $\barX$ whose homology class is given by $[\upF]\in\mathrm{H}_*(\barX,\Z)$. It is easily checked that there is a natural inclusion:
$$\widetilde{\mathcal{Z}_F(X)}\hookrightarrow\mathcal{Z}_{\upF}(\barX)$$
which realises $\widetilde{\mathcal{Z}_F(X)}$ as a Zariski open subset of $\mathcal{Z}_{\upF}(\barX)$. The compactification $\barX$ being K\"ahler, the irreducible components of its cycle space are compact, furthermore there are only finitely many components since the homology class is fixed. Since $\widetilde{\mathcal{Z}_F(X)}$ is a Zariski open set in a compact complex space, it has the homotopy type of a finite $CW$ complex. Thus  $\holom{\alpha_*}{\mathcal{Z}_F(X)}{A}$ is surjective by Theorem \ref{surjectivity}. 
If the Albanese map $\alpha_X$ is equidimensional, this already shows that $\alpha_X$ is locally trivial: every $\varphi$-fibre contains the image of an embedding $F \hookrightarrow X$, since the cohomology class is fixed the manifold $F$ is the whole fibre.

We will now prove by contradiction that $\alpha_X$ is equidimensional, and denote by $\emptyset \neq \Delta \subset A$ the locus where this is not the case.
Set $Z:=\fibre{\alpha_X}{\Delta}$ and consider the map $\holom{f:=\alpha_X|_Z}{Z}{A}$. The map $f$ is not surjective since
$\alpha_X$ is generically smooth. We claim that the induced cover $\upZ := Z \times_A \upA_{\uni}$ is Zariski open in a compact complex space,
which as before leads to a contradiction to Theorem \ref{surjectivity}.

{\em Proof of the claim.} Let $\overline{\univ}$ be the universal family over the unique component $\overline{\hilb}$ of $\chow{\barX}$ such that 
the general point corresponds to a general fibre of the fibration $\upX_{\uni} \rightarrow \upA_{\uni}$. There is a natural bimeromorphic
map $\overline{p}: \Gamma \rightarrow \barX$ and $\upZ \subset \upX_{\uni}$ corresponds to the points $x \in \upX$
such that $\fibre{\overline{p}}{x}$ has positive dimension. By Lemma \ref{lemmacompactifiable} this set is compactifiable.
\end{proof}

\begin{remarks}\label{final rem}
The arguments above are inspired by the proof of \cite[Thm.20]{KP12} and they rely heavily on the K\"ahler assumption on the compactification $\barX$:
in general the irreducible components of $\chow{\barX}$ are not compact if $\barX$ is merely a compact complex manifold.

However, Theorem \ref{surjectivity} and \cite[Thm.16]{KP12} are valid in the category of compact analytic spaces. This leads us to raise the following question: is there an equivalent statement of \cite[Thm.20]{KP12} in the compact complex category ? More precisely let us consider $X \rightarrow A$ a morphism between compact complex spaces and let us assume that $\upA_{univ}$ is contractible. If the induced cover $\upX$ is a Zariski open set in a compact complex manifold $\barX$, is $X \rightarrow A$ a fibre bundle ? We do not know any example where this is not the case.
\end{remarks}

\vspace*{0.3cm}
\noindent{\footnotesize\textsc{Beno\^it Claudon, Universit\'e de Lorraine, IECN, UMR 7502, B.P. 70239, F-54506 Vandoeuvre-l\`es-Nancy Cedex, France}\\
\emph{E-mail address:} \texttt{Benoit.Claudon@iecn.u-nancy.fr}}\\

\noindent{\footnotesize\textsc{Andreas H\"oring, Universit\'e Pierre et Marie Curie,
Institut de Math\'ematiques de Jussieu, \'Equipe Topologie et g\'eom\'etrie alg\'ebriques,
4 place Jussieu, 75252 Paris cedex 5, France}\\
\emph{E-mail address:} \texttt{hoering@math.jussieu.fr}}

\newpage

\appendix

\section{Abelianity, Iitaka and {\bf S} conjectures}

\begin{center}
\emph{by} Fr{\'e}d{\'e}ric \textsc{Campana}
\end{center}

\begin{abstract} 
In this appendix, we observe that Iitaka's conjecture fits in the more general
context of special manifolds, in which the relevant statements follow from
the particular cases of projective and simple manifolds.
\end{abstract}

Recall from \cite{C04} (for which we refer to the notions involved in the following):

\begin{conjecture}{\rm(Abelianity Conjecture, \cite[Conj.7.11]{C04})} Let $X$ be a special manifold. Then $\pi_1(X)$ is almost abelian.
\end{conjecture}

\begin{remark} 1. This conjecture is, at least, true for the linear
representations of the fundamental group, which have almost abelian image
\cite[th.7.8]{C04}.\\
2. The conjecture is also true up to dimension three \cite{CC}.
\end{remark}

Recall from \cite{C94} that for any compact K\" ahler manifold there exists a unique connected surjective almost holomorphic map: $\gamma_X:X\dasharrow \Gamma(X)$ such that its fibre $X_a$ through the general point $a\in X$ is the largest subspace $Y$ of $X$ through $a$ such that the image of $\pi_1(\hat Y)$ in $\pi_1(X)$ is finite, $\hat Y$ being the normalisation of $Y$. Then $\gamma d(X):=\dim(\Gamma(X))$ is called the $\gamma$-dimension of $X$. Thus $\gamma d(X)=0$ if and only if $\pi_1(X)$ is finite. When $\gamma d(X)=\dim(X)$, we say (as in \cite{CZ}) that $X$ is of $\pi_1$-general type.

\begin{conjecture}{\rm(Conjecture {\bf S})} Let $X$ be a special manifold with
$\gamma d(X)=\dim(X)$. Then some finite \'etale cover of $X$ is
bimeromorphic to a complex torus.
\end{conjecture}

\begin{remarks}\label{abgdmax} 1. Conversely, if some finite \'etale cover
of $X$ is bimeromorphic to a complex torus, $X$ is special with $\gamma
d(X)=\dim(X)$.

2. This conjecture {\bf S} implies the conjecture of Iitaka which claims the same conclusion as in S assuming that the universal cover of $X$ is $\C^n$.  The latter hypothesis
is indeed weaker (by the orbifold version of the theorem of Kobayashi-Ochiai
\cite[th.7.11 and 8.11]{C04}).

3. The Abelianity conjecture implies the conjecture {\bf S} (and so the
conjecture of Iitaka). Indeed, if $X$ is special of maximal
$\gamma$-dimension, its fundamental group is torsionfree and abelian (by
going to a suitable finite \'etale cover), if one assumes the Abelianity
conjecture. Its Albanese map is then surjective with connected fibres (by
\cite[th.5.3]{C04}) and induces an isomorphism on the first homology
groups. It has thus to be bimeromorphic, by the maximality of $\gamma
d(X)$.
\end{remarks}

\begin{definition} We shall say that $X$ is \emph{primitively special} if it is
special but not covered by special submanifolds of intermediate dimension
$0<d<\dim(X)$.
\end{definition}

\begin{lemma} $X$ is primitively special if and only if either:
\begin{enumerate}
\item $X$ is a rational or elliptic curve, or
\item $X$ is projective with $K_X$ pseudo-effective and with $\kappa(X)$ either $0$, or $-\infty$ (according to the abundance conjecture, this last case does not
exist), or
\item $X$ is simple (\emph{i.e.} not a curve and not covered by compact subspaces of intermediate dimensions).
\end{enumerate}
\end{lemma}

\begin{proof}
We assume that $X$ is primitively special with $n\geq 2$. We distinguish 3 cases.

Let us first assume $X$ to be projective. Then $\kappa(X)<n$, since $X$ is special. If
$0\leq\kappa(X)<n$, $X$ is covered by submanifolds having $\kappa=0$ and
intermediate dimension $n-\kappa(X)$ if $\kappa(X)>0$, which is impossible if $X$ is primitively special. Thus $\kappa(X)=0$ and $K_X$ is pseudo-effective. If $K_X$ is not pseudo-effective, then $X$ is uniruled, by \cite{BDPP} and \cite{MiMo}. The only remaining case is thus when $K_X$ is pseudo-effective and $\kappa(X)=-\infty$ (which Abundance conjecture claims not to exist). The projective case
is thus established.

Assume now that $0<a(X)<n$. Because the fibres of the algebraic
reduction are special, by \cite[th.2.39]{C04}, $X$ is not primitively
special.

Assume finally that $a(X)=0$, and that $X$ is not simple, but primitively
special. Let $Z_t$ be a covering family of $X$ by an analytic family of
subspaces which are generically irreducible and of intermediate
dimension $0<d<n$, chosen to be minimal. The generic member of this family is thus either of
general type, or special and primitively special. The second possibility
is excluded, since $X$ is primitively special. Thus $Z_t$ is of
general type. Let then $\varphi:X\dasharrow Y$ be the quotient by the
equivalence relation generated by the $Z_t's$ \cite{C81}. Its fibres
are projective, by \cite{C81}. Since $a(X)=0$, we have: $a(Y)=0$, and
$\dim(Y)>0$. From \cite{C85}, we get that the fibres of $\varphi$ are
almost-homogeneous, hence special. Contradiction since $X$ was assumed to
be primitively special. Thus $X$ is simple.
\end{proof}

\begin{remark}
The above argument is partially inspired by \cite{HPR}.
\end{remark} 

\begin{theorem}\label{reduction conj S}
The conjecture {\bf S} (and so the conjecture of Iitaka) is true if
{\bf S} is true whenever $X$ is primitively special. In particular, the conjecture {\bf S} is true if it is true in the projective and simple cases. 
\end{theorem}

\begin{proof}
Assume thus that $X$ is special, with $\gamma d(X)=n$. If $X$
is primitively special, we assume that {\bf S} is true. So assume that $X$ is
not primitively special and let $Z_t$ be a covering family of subspaces
which are special of intermediate dimension $0<d<n$. We may assume that
the generic member $Z_t$ is smooth, after suitable blow-ups of $X$. We may
thus assume that the conjecture holds true for the generic $Z_t$, by
working inductively on $n$. Thus the fundamental group of the generic
$Z_t$ is, in particular, almost abelian.

Let again $\varphi:X\dasharrow Y$ be the quotient by the equivalence relation
generated by the $Z_t's$. Its fibres are special \cite[th.3.3]{C04} (since they are connected by chains of special subspaces) and have maximal $\gamma$-dimension, since this is the case for $X$.\\

\noindent There are 2 cases, according to $m:=dim(Y)$.
\begin{enumerate}[(i)]
\item $m=0$. In this case, the fundamental group of $X$ is almost abelian, by
\cite{C98}, since $X$ is generated by connected chains of $Z_t's$, which
have almost abelian fundamental groups. Since $X$ has maximal $\gamma$-dimension, the
conjecture {\bf S} holds for $X$, by remark \ref{abgdmax} above.
\item $n>m>0$. In this case the generic fibres of $\phi$ have an \'etale cover
bimeromorphic to a torus. From \cite{Nak}, we conclude that $X$ has an
\'etale cover bimeromorphic to some $X'$ having a submersion $\psi:X'\to V$ on
a manifold $V$ of maximal $\gamma$-dimension, with fibres complex tori.
Because $V$ is special, since so is $X'$ \cite[th.5.12]{C04}, the conjecture {\bf S} is true for $V$, so that its fundamental group is almost abelian. From \cite{C98} again, we
deduce that the fundamental group of $X$ is almost abelian and that the
conjecture {\bf S} holds true for $X$, as claimed.
\end{enumerate}
\end{proof}

\vspace*{0.3cm}
\noindent{\footnotesize\textsc{Fr\'ed\'eric Campana, Universit\'e de Lorraine, IECN, UMR 7502, B.P. 70239, F-54506 Vandoeuvre-l\`es-Nancy Cedex, France}\\
\emph{E-mail address:} \texttt{Frederic.Campana@iecn.u-nancy.fr}}

\end{document}